\documentclass{asl}
\usepackage[utf8]{inputenc}
\usepackage[T1]{fontenc}
\usepackage{dsfont}
\usepackage{multicol}
\usepackage{color}
\usepackage{tikz-cd}

\def\tuple#1{\langle #1\rangle}
\def\set#1#2{\{#1\,|\,#2\}}
\def\f{\varphi}
\def\p{\psi}
\def\logic#1{{\mathrm {#1}}}
\def\frame#1{{{#1}}}

\def\modelt#1{{\mathrm {#1}}}

\def\classt#1{{\mathbf {#1}}}

\def\force{\Vdash}
\def\succ{\mbox{succ}}

\swapnumbers
\newtheorem{theorem}[subsection]{Theorem}
\newtheorem*{theorem*}{Theorem}
\newtheorem{proposition}[subsection]{Proposition}

\newtheorem{claim}[subsection]{Claim}
\newtheorem{lemma}[subsection]{Lemma}
\newtheorem{corollary}[subsection]{Corollary}

\theoremstyle{definition}
\newtheorem{definition}[subsection]{Definition}
\newtheorem*{notation*}{Notation}
\theoremstyle{remark}
\newtheorem*{remark}{Remark}

\newcommand{\fm}{\ensuremath{\mathit{Fm}}}

\title[Topological models of Intuitionistic logic]{Completely separably MAD families \\ and the modal logic of $\beta\omega$}

\author{Tom\'a\v{s} L\'avi\v{c}ka}
\revauthor{L\'avi\v{c}ka, Tom\'a\v{s}}
\address{The Czech Academy of Sciences\\ 
	Institute of Information Theory and Automation\\ 
	Pod vod\'{a}renskou v\v{e}\v{z}\'{i}~4\\
	182 07 Praha, Czech Republic
}
\email{lavickat@utia.cas.cz}

\author{Jonathan L. Verner}
\revauthor{Verner, Jonathan L.} 
\address{Department of Logic\\ 
	Faculty of Arts\\ 
	Charles University\\
	Palachovo n\'am. 2\\
	116 38 Praha 1 Czech Republic}
\email{jonathan.verner@ff.cuni.cz}
\thanks{The research of the first author was supported by project no.\ GA17-04630S of the Czech Science Foundation.
The research of the second author was supported by FWF-GAČR grant no.\ 17-33849L: Filters, Ultrafilters and Connections with Forcing.}

\begin{document}
\tikzcdset{
	arrow style=tikz,
	diagrams={>={Straight Barb[scale=0.8]}}
}

\begin{abstract}
We show in ZFC that the existence of completely separable maximal almost disjoint families of subsets of $\omega$ implies that the modal logic $\logic{S4.1.2}$  is complete with respect to the \v{C}ech-Stone compactification of the 
natural numbers, the space $\beta\omega$. In the same fashion we prove that the modal logic $\logic{S4}$ is complete with respect to the space {$\omega^*=\beta\omega\setminus\omega$}. This improves the results
of G.~Bezhanishvili and J.~Harding in \cite{Bezhanishvili:2009}, where the
authors prove these theorems under stronger assumptions (\(\mathfrak{a}=\mathfrak{c}\)). 
Our proof is also somewhat simpler. 
\end{abstract}
\maketitle

\section{Introduction}
Topological semantics appeared as one of the first semantics for both intuitionistic logic $\logic{IL}$ and modal logic $\logic{S4}$. First, in 1938, Tarski, in his pioneering work~\cite{Tarski:1938}, proved that intuitionistic logic is complete with respect to all topological spaces (propositions are interpreted as open sets). He also obtained the remarkable result that any dense-it-itself metrizable space (e.g. the Cantor space $\mathds{C}$ or the real line $\mathds{R}$), is complete for $\logic{IL}$. Later in 1944 he and McKinsey in their famous paper~\cite{Tarski:1944} proved an analogous results for the modal logic $\logic{S4}$. Recently in~\cite{Bezhanishvili:2009} G. Bezhanishvili and J. Harding described logical counterparts of the important spaces $\beta\omega$ and $\omega^*$, which turned out to correspond to modal logics $\logic{S4.1.2}=\logic{S4}+\Box\Diamond\f\leftrightarrow\Diamond\Box\f$ and $\logic{S4}$ respectively. However their proof needed the additional assumption that \(\mathfrak{a=c}\), i.e. the statement that every infinite maximal almost disjoint family of subsets of $\omega$ has cardinality of continuum $\mathfrak{c}$.

The main part of the proof of the completeness theorem resides in proving that there exists a surjective interior map from $\beta\omega$ (resp.\  $\omega^*)$ onto any finite tree with an additional top element (resp.\ any finite tree) where the topology of the target space is the upset topology. Additionally one also needs to guarantee that the fibers of the map are crowded in \(\beta\omega\setminus\omega\). In this paper we present a simpler construction of the map (and, consequently, prove the aforementioned completeness results). Moreover our proof relies on weaker assumptions, namely that there exists a completely separable maximal almost disjoint family of subsets of $\omega$. It is still an open problem whether completely separable maximal almost disjoint families exist in ZFC.

The content of this paper is organized as follows: in the preliminary section we recall basic notions concerning modal logics and their relational and topological completeness. We then describe the so called interior map strategy (see also \cite{Bezhanishvili:2009}) which we use to prove the main completeness theorems. The introductory part is concluded by recalling some basic facts about the topology of $\beta\omega$ and almost disjoint families. In the second section we construct the interior map and then we use this map in the last section to prove the main completeness results.

\section{Preliminaries}	
Modal logics\footnote{In general a \emph{logic} in a language \(\mathcal L\)
is just a structural consequence relation \(\vdash\) between sets of formulas of \(\mathcal L\) and formulas of \(\mathcal L\).} are presented using the standard modal language which consists of connectives of classical logic, a unary connective $\Box$ expressing necessity, and the $\Diamond$ connective expressing possibility, which is defined as $\Diamond\varphi = \neg \Box\neg\f$. The collection of all modal formulas is denoted as \fm.   Global modal logic $\logic{S4}$ can by axiomatized by any axiomatization of classical logic plus the following schemata of axioms for the modal operator:
\vskip0.3cm
\begin{center}
\begin{tabular}{|c|c|}
	(K) & 	$\Box(\f\to\p)\to(\Box\f\to\Box\p)$ \\ 
	(T)   & 	$\Box\f\to\f$ \\ 
  (4) & 	$\Box\f\to\Box\Box\f$ \\ 
  \end{tabular} 
\end{center}
\vskip0.3cm
and the Necessitation rule $\f/\Box\f$. We write $\Gamma\vdash_\logic{S4}\f$ if there is a proof of $\f$ from the set of formulas $\Gamma$. The modal logic $\logic{S4.1.2}$ has the same presentation with an additional axiom $$\Box\Diamond\f\leftrightarrow\Diamond\Box\f.$$
We write $\vdash_\logic{S4.1.2}$ for the provability relation in $\logic{S4.1.2}$.

\subsection{Topological semantics}
\begin{definition}
Let $\tuple{X,\tau}$ be a topological space. A \emph{topological model $\modelt{M}$} on $\frame{X}$ is a tuple $\tuple{\frame{X},\tau, \force}$,  where \(\force\) is a relation between
elements of \(X\) and modal formulas which, for every $x\in X$ and formulas $\varphi$, $\psi$, satisfies:
\begin{itemize}
	\item[(i)] 
	\(x\force\varphi\wedge\psi\) if and only if \(x\force\varphi\) and 
	\(x\force\psi\).
	\item[(ii)]
	\(x\force\varphi\vee\psi\) if and only if \(x\force\varphi\) or 
	\(x\force\psi\).
	\item[(iii)] 
	\(x\force\varphi\rightarrow\psi\) if and only if 
  \(x\not\force\varphi\) or \(x\force\psi\).
\item[(iv)] $x\force \neg\varphi$ if and only if $x\not\force\varphi$.
\item[(v)] $x\force \Box\f$ if and only if $\exists U\in\tau\, (x\in U\wedge \forall y\in U: y\force \f )$.
\end{itemize}
From the above and the definition of $\Diamond$ we also have:
\begin{itemize}
\item[(vi)] $x\force \Diamond\f$ if and only  $\forall U\in\tau\, (x\in U\to \exists y\in U: y\force \f )$.
\end{itemize}
We say that a (modal) formula $\f$ is \emph{valid in $\modelt{M}$} (or that \(\modelt{M}\) satisfies $\f$, written $\modelt{M}\models\f$) if for every $x\in X$, $x\force \f$. 
\end{definition}

Given a class $\classt{K}$ of topological spaces we define the corresponding consequence relation as follows: $\Gamma\models_\classt{K}\f$ if and only if for every  $\frame{X}\in\classt{K}$ and every topological 
model \(\modelt{M}\) on $X$ if \(\modelt{M}\) satisfies all formulas in $\Gamma$ then it also satisfies $\f$. 
When $\classt{K}=\{\frame{X}\}$ we write $\Gamma\models_\frame{X}\f$ (instead of $\Gamma\models_{\{\frame{X}\}}\f$).

\begin{remark} Given a topological model $\tuple{X,\tau,\force}$, let $||\f||=\set{x\in X}{x\force \f}$. Observe that $||\Box\f||=\mathrm{int}(||\f||)$ and $||\Diamond\f||=\overline{||\f||}$. In other words $\Box$ is interpreted as the topological interior operation and $\Diamond$ as the closure operation.
\end{remark}

\begin{definition} We say that a modal logic $\logic{L}$ is \emph{complete} with respect to the class of spaces \(\classt{K}\) if for every $\Gamma\cup\{\varphi\}\subseteq\fm$ we have $\Gamma\vdash_\logic{L}\varphi\Longleftrightarrow \Gamma\models_\classt{K}\varphi$.
\end{definition}

To describe completeness results for the logics $\logic{{S.4}}$ and $\logic{{S.4.1.2}}$, we first recall that a quasi-order $\tuple{T,\leq}$ is called a \textit{quasi-tree} if the canonical quotient order of $\tuple{T,\leq}$ is a tree.

\begin{notation*} We denote by \(\classt{T}\) (resp. \(\classt{T^*}\)) the class of all finite trees (resp.\ finite trees with an additional top element $*$) regarded as topological spaces with the topology given by upward closed sets
in the tree order. The notation \(\classt{QT}\) (resp. \(\classt{QT^*}\)) stands for the analogous classes corresponding
to quasi-trees instead of trees.
\end{notation*}

Then, the modal logic $\logic{S4}$ (resp.\ $\logic{S4.1.2}$) is known to be complete with respect to the class $\classt{QT}$ (resp.\ $\classt{QT^*}$). In symbols: for every $\Gamma\cup\{\f\}\subseteq\fm$

\begin{equation}\label{E:topcomplS4.1.2}
\Gamma\vdash_\logic{S4}\f\iff \Gamma\models_{\classt{QT}}\f.
\end{equation}
\begin{equation}\label{E:topcomplS4}
\Gamma\vdash_\logic{S4.1.2}\f\iff \Gamma\models_{\classt{QT^*}}\f,
\end{equation}

For proofs and more information on topological semantics for modal logics see~\cite{Aiello:2003} or~\cite{Benthem:2007}.

\subsection{The interior map strategy}\label{Sec:Interior}
We describe a general method (the \emph{interior map strategy}) which can be used to obtain new completeness results with respect to topological spaces from already existing ones (e.g. that $\logic{S4}$ is complete with respect to the Cantor space $\mathds{C}$ or the real line $\mathds{R}$, see \cite{Aiello:2003}).

Recall that a function between two topological spaces if called an \emph{interior map} 
if it is both continuous and open. It is not hard to prove that interior maps preserve topological consequence:

\begin{proposition}[\cite{Gabelaia:2001}]\label{P:interiormap}
Let $f$ be an interior map from $\frame{X}$ onto $\frame{Y}$ and $\Gamma\cup\{\f\}\subseteq\fm$. Then 
$\Gamma\models_\frame{X}\f$ implies $\Gamma\models_\frame{Y}\f$.
\end{proposition}

To prove a completeness result we typically need to show that if a formula is not provable, then
there is a model which serves as a counterexample. The above proposition allows us to transfer counterexamples
from one space to another. This is the idea of the \emph{interior map strategy}.

\subsection{The space \(\beta\omega\)}
Recall that the space \(\beta\omega\), the \v{C}ech-Stone compactification of 
the countable discrete space \(\omega\), can be naturally identified 
(via Stone duality), with the set of all ultrafilters on \(\omega\). 
The topology is given by basic (cl)open sets of the form
\[
\hat{A} = \{\mathcal U\in\beta\omega:A\in\mathcal U\},
\]
for some set \(A\subseteq\omega\). The space is a compact, extremally 
disconnected, zero-dimensional, topological space. The subspace 
consisting of all nonprincipal ultrafilters is denoted \(\omega^*\). It is
a compact zerodimensional space homeomorphic to the Stone space of the
algebra \(\mathcal P(\omega)/fin\).

The space \(\omega^*\) is very rich---it contains homeomorphic copies
of many spaces. The following theorem of P. Simon gives a rough idea what to expect.

\begin{theorem*}[\cite{Simon:1985}] Every extremally disconnected compact
space of weight \(\leq 2^\omega\) can be homeomorphically embedded into
\(\omega^*\) such that the image is a weak P-set.
\end{theorem*}

Also, since each \(\hat{A}\) is obviously isomorphic to \(\beta\omega\), these
copies are in fact dense.

\subsection{Trees}\label{Sub:trees}
We will also need some standard definitions regarding trees. A set \(T\)
together with a partial order \(\leq\) is a \emph{tree} if \(\{s\in T:s\leq t\}\)
is well-ordered by \(\leq\) for every \(t\in T\). Elements of the tree are
called \emph{nodes}. The \emph{height} of a node \(t\in T\), denoted \(ht(t)\),
is the rank of \(\{s\in T:s\leq t\ \&\ s\neq t\}\).
Given \(s\in T\) we write \(\succ_T(s)=\{t\in T:s\leq t\ \&\ ht(t)=ht(s)+1\}\)
for the set of immediate successors of \(s\) in \(T\) and \(T_s=\{t\in T:s\leq t\}\)
for the set of ``descendants'' of \(s\) (note that contrary to common usage
\(T_s\) does \emph{not} include the predecessors of \(s\)).

\subsection{Almost Disjoint Families}
Almost disjoint families will be the main tool in our construction of an interior map. Here we collect some basic facts about them.

\begin{definition} A family \(\mathcal A\subseteq\mathcal [\omega]^\omega\) is 
an \emph{almost disjoint family} (\emph{AD} for short) if \(|A\cap B|<\omega\) 
for each \(A,B\in\mathcal A\). An \emph{infinite} maximal (w.r.t. inclusion) AD
family is called a \emph{maximal almost disjoint family} (\emph{MAD} for short).
\end{definition}

Given an AD family \(\mathcal A\) we denote by \(\mathcal I(\mathcal A)\) the 
ideal on \(\omega\) generated by sets in \(\mathcal A\) together with all finite sets. 
It is natural to ask what are the possible sizes of a MAD family. It is not hard
to see that every MAD family must be uncountable and have size at most 
\(\mathfrak c\). The minimal possible size of a MAD family
\[
 \mathfrak a = \min\{|\mathcal A|:\mathcal A\ \mbox{is a MAD family}\}
\]
is called the \emph{almost disjoint number}. It can consistently be strictly 
smaller than \(\mathfrak c\). However, MAD families of size \(\mathfrak c\) 
always exist. To see this consider the family of branches of the full binary tree 
of height \(\omega\). This is clearly an AD family (on the countable set of nodes of
the binary tree) of size \(\mathfrak c\). Then, by the axiom of choice, it can 
be extended to a MAD family.

The construction of the surjective map in \cite{Bezhanishvili:2009} uses the 
assumption \(\mathfrak a=\mathfrak c\) which guarantees that any small AD family
cannot be maximal. This allows them to recursively construct a MAD family
dealing with \(\mathfrak c\)-many tasks. If \(\mathfrak a<\mathfrak c\), then
the construction might run into a dead-end by constructing a MAD family before 
all tasks are met. 

We will use a different approach, relying on the following concept:

\begin{definition} An AD family \(\mathcal A\) is \emph{completely separable}
if for every \(X\not\in\mathcal I(\mathcal A)\) either \(X\) is almost 
disjoint with each element of \(\mathcal A\) or there is \(A\in\mathcal A\)
such that \(A\subseteq^* X\) or, equivalently, there are \(\mathfrak c\)-may
such \(A\)s (recall that \(A\subseteq^*X\) means that \(|A\setminus X|<\omega\)).
\end{definition}

Notice that if \(\mathcal A\) is a MAD family then the first alternative cannot
hold. A completely separable MAD family will be instrumental in our proof. 
While it is easy to construct a completely separable AD family by transfinite 
recursion, it is a long-standing open question whether there is, 
in ZFC, a completely separable MAD family. Their existence easily follows from 
\(\mathfrak a=\mathfrak c\). Recently (see \cite{Shelah:2011:MAD}, and others), they
have been shown to exist under various much weaker assumption, e.g. if \(\mathfrak c\leq\aleph_\omega\),
but it is conceivable, that their existence is outright provable in ZFC. 
For more on this topic see, e.g., the survey paper \cite{Hrusak:2007}.

Finally, we finish with a simple technical lemma (observation, really), which will be important in the next section.

\begin{lemma}\label{lemma:crowded}
Assume \(\mathcal A\) is an infinite AD family of infinite subsets of \(X\subseteq\omega\). Then the set
\[
 Z=\omega^*\cap\hat{X}\setminus\bigcup\left\{\hat{A}:A\in\mathcal A\right\}
\]
is a compact subspace of \(\omega^*\) without isolated points.
\end{lemma}
\begin{proof} The set is closed and hence compact. To see that it has no
isolated points let \(p\in Z\) and \(\hat{B}\) a neighbourhood of \(p\). There are two cases. Either
\[
p\not\in\overline{\bigcup\left\{\hat{A}:A\in\mathcal A\right\}}.
\]
Then there must be an \(C\in p\) such that \(C\) is almost disjoint from \(\mathcal A\). Then \(\hat{B}\cap \hat{C}\cap \hat{X}\cap\omega^*\subseteq Z\) is a clopen neighbourhood of \(p\) in \(\omega^*\) so, in particular, \(p\) is not
an isolated point.

Dealing with the second case assume that 
\[
 p\in\overline{\bigcup\left\{\hat{A}:A\in\mathcal A\right\}}\setminus\bigcup\left\{\hat{A}:A\in\mathcal A\right\}
\]
It follows that \(B\) can't be almost disjoint from \(\mathcal A\) and can't be almost covered by finitely many elements of \(\mathcal A\) either. Pick a countable family \(\{A_n:n<\omega\}\subseteq\mathcal A\) such that \(B^\prime_n=B\cap A_n\) is
infinite for each \(n\). Let 
\[
 B_n=B^\prime_n\setminus\bigcup_{i<n}B^\prime_i
\]
and partition each \(B_n\) into two disjoint infinite subsets \(C_n\) and \(D_n\) and let
\[
 C=\bigcup_{n<\omega} C_n,\quad D=\bigcup_{n<\omega}D_n.
\]
Then \(C,D\subseteq X\cap B\) are disjoint so at most one can be an element of \(p\). Without loss of generality assume that \(C\not\in p\). Notice that \(C\) cannot be almost covered by finitely many elements of \(\mathcal A\) (because it hits each \(A_n\) in an infinite set and \(\mathcal A\) is almost disjoint) so 
\[
 E=\hat{C}\setminus\bigcup\left\{\hat{A}:A\in\mathcal A\right\}\neq\emptyset.
\]
In particular there is \(q\in E\cap\hat{B}\). Since \(E\subseteq C\not\in p\) necessarily \(q\neq p\). So
\(|\hat{B}\cap Z|\geq 2\). Since \(\hat{B}\) was an arbitrary clopen neighbourhood of \(p\) this finishes the proof 
that \(p\) is not isolated in \(Z\).
\end{proof}

\section{The interior map construction}

The aim of this section is to prove the following theorem:

\begin{theorem} \label{Thm:tree} 
Assume there is a completely separable MAD family. Then for every $T\in\classt{T^*}$ there is a surjective interior map $f$ from $\beta\omega$ onto $T$.
\end{theorem}

\begin{proof}
For each \(t\in T\) we will recursively construct an almost disjoint family \(\mathcal A_t\) as follows. 

First, for notational convenience, we introduce a unique formal predecessor of \(\emptyset\) in \(T\) and we denote it by \(-1\). For this element we let \(\mathcal A_{-1}=\{\omega\}\). Assume now that \(\mathcal A_s\) has been constructed for some non-leaf node \(s\). Partition each element \(A\in\mathcal A_s\) arbitrarily into \(|succ_T(s)|\)-many infinite pieces:
\[
 A = \bigcup_{t\in succ_T(s)} A(t),
\]
and for each \(t\in succ_T(s)\) and \(A\in\mathcal A_s\) let \(\mathcal A(t,A)\) be a completely separable MAD family on \(A(t)\), which exists by our assumption. Finally, for \(t\in succ_T(s)\) let 
\[
 \mathcal A_t=\bigcup_{A\in\mathcal A_s} \mathcal A(t, A).
\]
This finishes the recursive construction.

Next we define the map \(f\). The isolated points in \(\beta\omega\) are mapped to the top element, and for 
each \(t\in succ_T(s)\) we map 
\begin{equation}\label{eq:nonleaf}
  \omega^*\cap\bigcup_{A\in\mathcal A_s} \hat{A}(t)\setminus\bigcup\left\{\hat{B}:B\in\mathcal A(t,A)\right\}
\end{equation}
onto \(t\). Finally, if \(t\in succ_T(s)\) is a leaf node, we also map
\begin{equation}\label{eq:leaf}
 \omega^*\cap\bigcup\left\{\hat{B}:B\in\mathcal A(t,A)\right\}
\end{equation}
onto \(t\).
It is clear that the map is well-defined.
\begin{claim} If \(t\in succ_T(s)\) then 
\[
  f^{-1}[T_t] = \omega^*\cap \bigcup_{A\in\mathcal A_s} \hat{A}(t)
\]
\end{claim}
\begin{proof}[proof of Claim]
We prove the claim by induction on the distance of \(t\) from the top. For leaf
nodes (distance 0), the claim follows immediately from \ref{eq:nonleaf}.
Note that
\[
f^{-1}[T_t] = \left(f^{-1}(t)\cup \bigcup_{t^\prime\in succ_T(t)}f^{-1}[T_{t^\prime}].
 \right)
\]
Expanding the definitions we get
\begin{equation}\label{eq:expanded}
 f^{-1}[T_t] = \omega^*\cap\left(
    \bigcup_{A\in\mathcal A_s} \hat{A}(t)\setminus\bigcup\left\{\hat{B}:B\in\mathcal A(t,A)\right\}\cup
    \bigcup_{t^\prime\in succ_T(t)}f^{-1}[T_{t^\prime}]
 \right),
\end{equation}
finally, by the inductive assumption (applied to \(s=t, t=t^\prime\)) and by the
definition of \(\mathcal A_t\), we have
\begin{equation}\label{eq:inductive}
 f^{-1}[T_{t^\prime}] = \omega^*\cap\bigcup_{B\in \mathcal A_t}\hat{B}(t) 
		    = \omega^*\cap\bigcup_{A\in\mathcal A_s}\bigcup_{B\in\mathcal A(t,A)}\hat{B}(t)
\end{equation}
Putting together (\ref{eq:expanded}) and (\ref{eq:inductive}) gives the claim.
\end{proof}
From the claim it immediately follows that the map \(f\) is continuous. It is clearly onto so it remains to show that it is open. First it is clear that the image of every open set contains the top element. Now let \(X\in[\omega]^\omega\) and assume that \(t\in f[\hat{X}]\) for some \(t\in T\). We need to show that \(T_t\subseteq f[\hat{X}]\). Again, we prove this by downward induction on the distance of \(t\) from the top. For leaf nodes there is nothing to prove. So assume \(s\in T\)
and that for \(t\in succ_T(s)\) and any \(Y\) with \(t\in f[\hat{Y}]\) we have, by the inductive assumption, \(T_t\subseteq f[\hat{Y}]\). It is then enough to show that \(t\in f[X]\) for every \(t\in succ_T(s)\). Pick a \(p\in\hat{X}\)
such that \(f(p)=s\). Then, by (\ref{eq:nonleaf}), we have
\[
 p\in\bigcup_{A\in\mathcal A_{s^-}} \hat{A}(s)\setminus\bigcup\left\{\hat{B}:B\in\mathcal A(s,A)\right\},
\]
where \(s^-\) is the predecessor of \(s\) in \(T\) (the predecessor of the root element being the formal element \(-1\)).
In particular for some \(A\in\mathcal A_{s^-}\), \(p\in \hat{A}(s)\). Let \(Y=X\cap A(s)\). Since \(f(p)=s\), 
necessarily 
\[
 p\not\in\bigcup\left\{\hat{B}:B\in\mathcal A(s,A(s))\right\}.
\]
It follows that \(Y\) cannot be covered by finitely many elements of the completely separable MAD \(\mathcal A(s,A(s))\) on \(A(s)\). By complete separability, there must be \(B\in\mathcal A(s,A(s))\) such that \(B\subseteq^* Y\subseteq X\). So we have \(B(t)\subseteq B\) and \(t\in f[\hat{B}(t)]\) so, a fortiori, \(t\in f[\hat{X}]\).
\end{proof}

\begin{corollary} \label{Cor:toppedquasitree} 
For every $X\in\classt{QT^*}$ there is a surjective interior map $f$ from $\beta\omega$ onto $X$.
\end{corollary}
\begin{proof} This is precisely the Main Lemma and Corollary~4.10 from \cite{Bezhanishvili:2009}.
The idea is the following: let $X\in\classt{QT^*}$ be given. Then there is a projection map $\pi$ from $X$ onto its quotient order $X/_\approx$. This projection is an interior map with respect to the corresponding upset topologies. By Theorem~\ref{Thm:tree} there is an interior onto map $f:\beta\omega\to X/_\approx$.
We prove there is a surjective interior map $g\colon\beta\omega\to X$ such that the following commute:
\begin{center}
\begin{tikzcd}
\beta\omega \arrow[dashrightarrow ,d, "g"'] \arrow[r,two heads, "f"] & X/_\approx \\
X \arrow[ ur,two heads, "\pi"']& 
\end{tikzcd}
\end{center}
Such a map exists because for each $*\neq x\in X/_\approx$ the subspace $f^{-1}(x)$ is $k$-resolvable (i.e.\ it is a union of $k$-many disjoint dense subsets) for every natural number $k$, in particular for $k=|\pi^{-1}(x)|$. Indeed looking at (\ref{eq:nonleaf}) in the definition of \(f\) in Theorem~\ref{Thm:tree} and lemma \ref{lemma:crowded}, $f^{-1}(x)$ is locally compact and has no isolated points. It is obviously Hausdorff so it is $k$-resolvable by \cite[see comments on p. 332]{Hewitt:1943}.
\end{proof}

A straightforward modification of the above proofs gives (just restricting the map to \(\omega^*\) whose image is
the tree without the top element):

\begin{corollary}\label{Cor:quasitree}
For every $X\in\classt{QT}$ there is a surjective interior map $f$ from $\omega^*$ onto $X$.
\end{corollary}

\section{Completeness results}

\begin{theorem}\label{Thm:completenessBetaOmega}
The Modal logic $\logic{S4.1.2}$ is the logic of the space $\beta\omega$.
\end{theorem}
\begin{proof}
Completeness: suppose $\Gamma\not\vdash_{S4.1.2}\f$, then by~\eqref{E:topcomplS4.1.2} we obtain $\Gamma\not\models_\classt{QT^*}\f$. In other words for some $X\in\classt{QT^*}$ we have $\Gamma\not\models_X\f$. By Proposition~\ref{P:interiormap} and Corollary~\ref{Cor:toppedquasitree} we conclude that $\Gamma\not\models_{\beta\omega}\f$.

Soundness: From $\Gamma\vdash_{S4.1.2}\f$ we want to prove $\Gamma\models_{\beta\omega}\f$. This is an easy inductive argument on the complexity of the proof. In more detail: the axiom $\Box\Diamond \f\to\Diamond\Box\f$ is valid in every space with a dense set of isolated points (see \cite{Bezhanishvilli:2003}). The axiom $\Diamond\Box\f\to\Box\Diamond\f$ is valid precisely in extremally disconnected spaces (see \cite{Gabelaia:2001}).
\end{proof}

\begin{theorem} The Modal logic $\logic{S4}$ is the logic of the space $\omega^*$.
\end{theorem}
\begin{proof}
Completeness follows by essentially the same argument as in the proof of Theorem~\ref{Thm:completenessBetaOmega}: we use~\eqref{E:topcomplS4}, Proposition~\ref{P:interiormap}, and Corollary~\ref{Cor:quasitree}.

Soundness follows from the fact that every topological space is sound w.r.t.\  the modal logic $\logic{S4}$.
\end{proof}

\section{Conclusion}	
Let us mention that the completeness results for the modal logic $\logic{{S4}}$ (and $\logic{{S4.1.2}}$, resp.) obtained in this  paper translate into completeness results for intutionistic logic (and, resp., the Jankov's logic $\logic{{KC}}$---i.e.\ intuitionistic logic + the axiom of weak excluded middle $\neg\varphi\vee \neg\neg\varphi$). For more details see~\cite{Bezhanishvili:2009}. The results can also be used to weaken the assumptions of many results from~\cite{Bezhanishvili:2012}, where the authors discuss the logics of stone spaces in general. For other topological completeness results for modal logics see, e.g., \cite{Bezhanishvili:2015,Bezhanishvili:2017}. 

While we needed to assume for our construction that there is a completely separable MAD family, we still think that this is not necessary. Unfortunately, we only have partial results in this direction: we know that Theorem \ref{Thm:tree} holds for trees of height $2$ without additional assumptions---the proof is analogous but in this case it is enough to consider $k$-partitionable MAD families, which are known to exists in ZFC (\cite{Douwen:1993b}); the theorem also holds for all binary trees (though with a very different proof using the fact, due to B.~Balcar and P.~Vojtáš (\cite{Balcar:1980b}, that
every ultrafilter has an almost disjoint refinement\footnote{The fact that using almost disjoint refinements works is not entirely unexpected because a MAD family \(\mathcal A\) is completely separable iff it is an almost disjoint refinement of \(\mathcal I^+(\mathcal A)\).}). Of course, if it turns out that completely separable MAD families 
exist in ZFC, our construction will be a ZFC construction.

\bibliographystyle{asl}
\bibliography{main}
\end{document}